\documentclass[reqno]{amsart}

\usepackage[latin9]{inputenc} 
\usepackage{amssymb}
\usepackage[initials]{amsrefs} 
\usepackage{pst-all}  

\numberwithin{subsection}{section}
\numberwithin{equation}{section}

\theoremstyle{plain}
 \newtheorem {theorem}    {Theorem}[section]
 \newtheorem {proposition}[theorem]{Proposition}
 \newtheorem {lemma}      [theorem]{Lemma}
 \newtheorem {corollary}  [theorem]{Corollary}
\theoremstyle{definition}
 \newtheorem {definition} [theorem]{Definition}

\theoremstyle{remark}
 \newtheorem {remark}     [theorem]{Remark}

\newcommand{\bs}{\boldsymbol}

\newcommand{\mathI}{\mathbb{I}}
\newcommand{\mathC}{\mathbb{C}}
\newcommand{\mathN}{\mathbb{N}}

\newcommand{\calH} {\mathcal{H}}
\newcommand{\Poiss}{K}      				  

\newcommand\hor{\mathfrak{h}}

\newcommand{\vertiii}[1]{{\left\vert\kern-0.25ex\left\vert\kern-0.25ex\left\vert #1 
    \right\vert\kern-0.25ex\right\vert\kern-0.25ex\right\vert}}

\DeclareMathOperator{\s}{\sigma}

\DeclareMathOperator{\dist}{dist}

\DeclareMathOperator{\Poisstr}{\bs K_\lambda}  

\DeclareMathOperator{\Real} {{\rm Re}}
\DeclareMathOperator{\Imag }{{\rm Im}}

\newcommand\C{\mathbb C}

\newcommand\la{\lambda}

\newcommand\spec{\mathsf{spec}}

\newcommand\T{U}

\begin{document}

\title[Universal properties of the Laplacian on homogeneous trees]
{Universal properties of the isotropic Laplace operator on homogeneous trees}
\author[J.~M. Cohen]{Joel M. Cohen}
\address{Department of Mathematics\\
University of Maryland\\
College Park, Maryland 20742\\USA} \email{jcohen@umd.edu}
\author[M. Pagliacci]{Mauro Pagliacci}
\address{Dipartimento di Economia\\
Universit\`a degli Studi di Perugia\\ I-06123 Perugia, Italy} \email{mauro.pagliacci@unipg.it}
\author[M.A. Picardello]{Massimo A. Picardello}
\address{Dipartimento di Matematica\\ 
Universit\`a di Roma ``Tor Vergata''\\
I-00133 Rome, Italy}
\email{picard@mat.uniroma2.it}

\thanks{The first  author acknowledge partial support from the University of Perugia during the preparation of this article.
The last author acknowledges support by MIUR Excellence Departments Project awarded to
the Department of Mathematics, Tor Vergata University of Rome, CUP E83C18000100006, and by Istituto Nazionale di Alta Matematica, Gruppo GNAFA}

\begin{abstract} 
Let $P$ be the isotropic nearest neighbor transition operator $P$ on a homogeneous tree. We consider the $\lambda$-eigenfunctions of $P$ for $\lambda$ outside its $\ell^2$-spectrum $\spec(P)$, i.e.,  the eigenfunctions with eigenvalue  $\gamma=\lambda - 1$
 of the Laplace operator $\Delta=P-\mathI$, and also the $\lambda-$polyharmonic functions, that is, the union of the kernels of $(\Delta-\gamma\mathI)^n$. We prove that,  on
a suitable Banach space generated by the $\lambda-$polyharmonic functions, the operator $e^{\Delta-\gamma\mathI}$ is hypercyclic, although $\Delta-\gamma\mathI$ is not. 
\end{abstract}

\maketitle

\section{Introduction}\label{Sec:introl} 
{\bf A slightly less detailed version of this article is in print in Advances in Mathematics \cite{CPP}.}

In this paper we construct  examples of hypercyclic operators on Banach spaces of functions defined on the vertices of a homogeneous tree $T$, that is, operators that satisfy the universal property that some of their orbits are dense everywhere. The Banach spaces are generated by $\gamma$-eigenfunctions of  the isotropic Laplacian $\Delta$ on $T$ and more generally by the kernels of powers of $\Delta-\gamma\mathI$; the operators are related to the heat semigroup generated by $\Delta-\gamma\mathI$. The spectral theory of $\Delta$ was studied long ago  \cites{FtP,CC}; for the heat semigroup, see \cite{PP}.

Hypercyclic operators have been studied extensively in the context of Laplace operators in continuous settings and the heat equation (see, for instance, \cites{ADB, H, HS}) and of universal properties \cite{G}. Following \cite{BS}, hypercyclic operators have been studied in the context of financial models, like the Black-Scholes model of \cites{BS, EGG}). Universal properties of the eigenfunctions of the Laplace operator on a large class of trees have been recently studied
in \cites{ANP-JMAA,ANP}. 

Trees have a simple combinatorial and metric structure. This makes them a natural crossroad of ideas and tools from Markov chains and probability theory, potential theory, harmonic analysis, eigenspaces of the Laplacian,  representations of groups of automorphisms, spherical functions, horospherical analysis and Radon transforms and hypercyclic operators (see \cites{FtP,Woess} and references therein).
Our construction is based on most of these tools, mainly,  on functional analysis, potential theory and the Poisson transform on trees, and a result of Herzog and Schmoeger on hypercyclic operators \cite{HS}.

Here is a more extensive outline {\color{black}{(see \cite{Ca} for more details)}}. A homogeneous tree $T$ is a graph without loops where every vertex has the same number $q+1$ of neighbors; the number $q$ is called the \emph{degree} of $T$ and is assumed larger than 1.
 We consider the isotropic nearest neighbor stochastic  transition operator $P$: that is, $p(v,u)=1/(q+1)$ if $v$ and $u$ are adjacent vertices and 0 otherwise. The boundary $\partial T$ is defined as the set of all infinite geodesic paths, i.e., chains of consecutively adjacent vertices, starting at a fixed root vertex $o$ (or equivalently, in a root-free way, as the set of equivalence classes of geodesic paths starting at any vertex,   equipped with the equivalence relation that identifies two paths if they {\color{black}{definitely}} merge; in particular, $\partial T$ does not depend on the choice of root vertex $o$). When $u$ and $v$ are adjacent we write $u\sim v$.

The \emph{Laplace operator} is $\Delta=P-\mathI$. This operator acts on functions $f$ on $T$ by the rule $\Delta f(v)=d\frac1{q+1}\biggl(\sum_{w:\,w\sim v}  f(w)\biggr)-f(v)$. The  functions that satisfy the mean value property $\Delta f=0$ are called \emph{harmonic}. The harmonic functions form the zero-eigenspace of $\Delta$. 
A Poisson boundary representation of {\color{black}{all}} eigenfunctions of group-invariant nearest neighbor transition operators on homogeneous trees was studied  in \cite{FtP} and its references. {\color{black}{(The Poisson representation was previously introduced in \cite{Ca} for harmonic functions only,
however, this reference, that deals with the more general case of non-homogeneous trees, considers nearest-neighbor positive transition operators that are not necessarily Markovian, so, by a renormalization, one could transport the Poisson representation of \cite{Ca} to all positive eigenvalues, but not to all complex eigenvalues). }}

In addition to eigenfunctions, we consider polyharmonic functions. A function $f$ is polyharmonic of order $m>0$ if $\Delta^m f=0$, and $\gamma$-polyharmonic if $(\Delta-\gamma\mathI)^m f=0$. Polyharmonic functions on homogeneous trees with isotropic transition operators have been studied in \cite{CCGS}. \ More recently, a boundary representation for polyharmonic and $\gamma$-polyharmonic functions has been obtained in \cites{PW-PotAn}\ 
for non-homogeneous trees, even non-locally finite and with non-isotropic transition operators.

\section{The Poisson transform on trees}
We introduce some preliminaries on potential theory on homogeneous trees taken from \cite{Ca}.

\begin{remark}[Eigenfunctions of the Laplacian]\label{rem:Poisson_transform_of_distributions} Let $P$ be the nearest neighbor 
isotropic transition operator on the homogeneous tree $T$ of degree $q$.    For every $x\in T$ we set $|x|=\dist(x,o)$, i.e. the number of edges from $x$ to $o$.
 We denote by $\spec(P)$ its spectrum on $\ell^{2}(T)$; the spectral radius is $\rho=\dfrac{2\sqrt{q}}{q+1}$.

Poisson boundaries of trees were introduced in \cite{Ca}*{Chapter I}.
We have already mentioned that
the \emph{boundary} $\partial T$ of $T$ can be defined as the set of infinite geodesic paths starting at $o$. The \emph{sectors} $S_u=\{v\colon u$  belongs to the geodesic path from $o$ to $v\}$ generate a topology on $T\cup \partial T$ that makes this space compact. Let us assign a probability measure on $\partial T$ by the rule
\[
\nu_o(S_u)=\frac 1{\#\{x\colon |x|=|u|\}}
\]
For all $u\in T$ and $\xi\in\partial T$ we denote by $u\wedge \xi$ the last vertex in common in the finite geodesic path from $o$ to $u$ and the infinite geodesic path $v_0$ to $\xi$.
The \emph{horospherical index} of $x$ and $\xi$ with respect to $o$ is
 $\hor(x,\xi)=
|x\wedge\xi|-\dist (x,\,x\wedge\xi)=2 |x\wedge\xi|-|x|$.  
Note that
\begin{equation}\label {eq:max_of_horospherical_index_on_the_boundary}
-|x| \leqslant \hor(x,\xi) \leqslant |x|.
\end{equation}
For every fixed $\xi$, the maximum value $n$ of $\hor(x,\xi)$ over all vertices $x$ with $|x|=n$ is 
attained at the vertex $x$ that lies in the geodesic path $\xi$.
A function $h$ on $T$ is harmonic if is an eigenfunction of $P$ with eigenvalue 1, that is if it satisfies the nearest-neighbor average property $Ph(x):=\sum_{y\sim x} p(x,y)\,h(y)= h(x)$. 
The \emph{Poisson kernel} %
\citelist{\cite{Ca}*{Sect. 4.5, formula (4.50)}
\cite
{FtP}*{Chapter 3, Sect. 2, formula (1)}
}
%
is
 \begin{equation*}
  \Poiss(x,\xi)= q^{\hor(x,\xi)},
  \end{equation*}
and is a harmonic function normalized by 
$\Poiss(o,\xi)=1$
for every $\xi$; it is a minimal positive harmonic function, that is, an extreme point in the cone of normalized positive harmonic functions.
The Poisson representation theorem \cite{Ca}*{Prop. A.4} states that every harmonic function can be represented as 
$h(x)=\int_{\partial T}  \Poiss(x,\xi)\, d\sigma^h(\xi)$ where $\sigma^h$ is a finitely additive measure (or \emph{distribution}) on $\partial T$.

It was shown in \cites{Mantero&Zappa,FtP} that every eigenfunction of $P$ with eigenvalue $\lambda\in\mathC$ is given by
\[
h(x)=\int_{\partial T}  \Poiss(x,\xi |\,\lambda)\, d\sigma^h(\xi)
\]
where 
$\sigma^h$ is a distribution on $\partial T$ and 
 \begin{equation}\label{eq:generalized_homogeneous_isotropic_Martin_kernel}
  \Poiss(x,\xi|\,\la)= q^{z\,\hor(x,\xi)},
  \end{equation}
with $z\in\mathC$ and
$
\lambda=\gamma(z)$ given by 
\begin{equation}\label{eq:the_eigenvalue_map}
\gamma(z)=(q^z+q^{1-z})/(q+1)
\end{equation}
where we can restrict attention to the region
\begin{equation}\label{eq:Re(z)>=1/2}
J=\left\{z\colon \Real z\geqslant \frac12, \qquad | \Imag{z} |\leqslant \frac{\pi}{\ln q}\right\} .
\end{equation}
The expression of $\gamma$ in \eqref{eq:the_eigenvalue_map} 
was 
introduced in \cites{FtP-JFA,FtP}.
The map $\gamma:\mathC\to\mathC$ is surjective, 
periodic along the imaginary direction with period $|\Imag{z}|\leqslant 2\pi/\ln q$,
 and in the region $J$ is bijective onto $\mathC$.
In particular, $ \Poiss(x,\xi|\,\la)$ is an eigenfunction of $P$ with eigenvalue $\lambda$, normalized by the rule $ \Poiss(o,\xi|\,\la)=1$.
\end{remark}

The boundary representation of polyharmonic functions on trees was developed in \cite{CCGS}, and more generally in \cite {PW-PotAn}. In particular \cite {PW-PotAn}*{Theorem 5.3, Corollary 5.4}:
\begin{enumerate}
\item [$(i)$]
Let $P$ be a stochastic nearest-neighbor transition operator an a tree (not necessarily homogeneous). For all
 $\lambda\notin\spec(P)$, $x\in T$, $\xi\in\partial T$,
\begin{equation}\label {eq:action_of_P-lambda_on_lambda-polyharmonics}
(P - \lambda\, \mathI ) \Poiss^{(r)}(x,\xi|\la) = (-1)^r r\,  \Poiss^{(r-1)}(x,\xi|\lambda)\,.
\end{equation}
where $\Poiss^{(r)}=\dfrac{\partial^r}{\partial \lambda^r}K$.
%
Every $\la$-polyharmonic function 
$h$ of order $n$ of a nearest-neigbor transient $P$ has {\color{black}{the }} integral representation
\begin{equation} \label {eq:Poisson_representation_of_lambda-polyharmonic}
h(x) = \sum_{r=0}^{n-1} \int_{\partial T} \Poiss^{(r)}(x,\xi|\,\la)\,d\sigma_r(\xi)\,,
\end{equation}
where the collection of distributions $(\sigma_0\,,\dots, \sigma_{n-1})$  is uniquely determined by $h$.  
Conversely, every function which has an integral representation as above, with $\sigma_{n-1}\not\equiv 0$,
is $\la$-polyharmonic of order $n$ for $P$. An eigenfunction of the Laplacian is non-negative if and 
only if the associated distribution is a non-negative measure, hence $\sigma$-additive.

\item [$(ii)$]
If $P$ is the isotropic nearest-neighbor operator on the homogeneous tree
 $T = T_q$, with spectral radius $\rho$, and $\la \in \C \setminus [-\rho\,,\,\rho]$,
then every  $\la$-polyharmonic function $h$ of order $n$ has {\color{black}{the }} 
integral representation
$$
h(x) = \sum_{k=0}^{n-1} \int_{\partial T} \Poiss(x,\xi|\,\la)\,\hor(x,\xi)^k\, d\bar\sigma_k(\xi)\,,
$$
where 
the collection of distributions $(\bar \sigma_0\,,\dots, \bar\sigma_{n-1})$ in the sense
of \eqref{eq:Poisson_representation_of_lambda-polyharmonic} is uniquely determined by $h$.  
Moreover, there exist complex numbers $ a_{k,r}(\la)$ such that
%
\begin{equation} \label{eq:diagonal_terms}
 a_{r,r}(\la) \neq 0
 \end{equation}
and
\begin{equation}\label {eq:derivatives_of_Poisson_kernels_in_terms_of_Poisson_kernels_times_horsph.indices}
K^{(r)}(x,\xi|\la) = K(x,\xi|\la)\, \sum_{k=1}^r \hor(x,\xi)^k \, a_{k,r}(\la).
\end{equation}
Finally, 
\begin{equation} \label{eq:sigma-bar_in_terms_of_sigmas}
\bar \sigma_k = \sum_{r=k}^{n-1} a_{k,r}(\la)\,\sigma_r.
\end{equation}
\end{enumerate}

\begin{corollary}
The upper triangular matrix 
$A_{n-1}(\la) = $ $\bigl( a_{k,r}(\la) \bigr)_{1 \leqslant k \leqslant r \leqslant n-1}$
is invertible by
\eqref {eq:diagonal_terms}. Let us denote its norm by $\mathcal{A}(\lambda)$. 
Then
\begin{enumerate}
\item[$(i)$] if $\| \bar\sigma_k\|, \| \sigma_r\|$ denote the measure norms on $\partial T$, then
\begin{equation} \label{eq:inequality_between_norm_of_sigma_and_sigma_bar}
\| \bar \sigma_k  \| \leqslant \mathcal{A}(\lambda) \biggl( \sum_{r=k}^{n-1} \| \sigma_r \|^2 \biggr)^\frac12.
 \end{equation}
\item[$(ii)$] For $\lambda\notin\spec(P)$, more precisely for $\lambda=\gamma(z)$ with $z$ in the region  $J$ of \eqref{eq:Re(z)>=1/2},
\begin{equation} \label {eq:inequality_between_norm_of_derivative_of_Poisson_kernel_and_Poisson_kernels_times_horsp.indices}
\max_{\xi\in \partial T} |K^{(r)}(x,\xi|\la)| = q^{|x|\Real z}\, \mathcal{A}(\lambda)\, \biggl( \sum_{k=1}^r |x|^{2k} \biggr)^\frac12.
\end{equation}
\end{enumerate}
\end{corollary}

\begin{proof}
Part $(i)$ is an immediate consequence of \eqref{eq:sigma-bar_in_terms_of_sigmas}.
Part $(ii)$ follows similarly from \eqref{eq:derivatives_of_Poisson_kernels_in_terms_of_Poisson_kernels_times_horsph.indices}
thanks to
\eqref{eq:max_of_horospherical_index_on_the_boundary},
\eqref{eq:generalized_homogeneous_isotropic_Martin_kernel} 
and Cauchy--Schwarz inequality, because, if $\boldsymbol {e_1}=(1,0,\dotsc,0)$, then
$\| (a_{1,r},\dotsc,a_{r,r}\|=\|A_{n-1}(\lambda) \boldsymbol {e_1}\| \leqslant \mathcal A(\lambda)\|$. 
Note that \eqref{eq:inequality_between_norm_of_derivative_of_Poisson_kernel_and_Poisson_kernels_times_horsp.indices} makes sense since $z$ is a function of $\lambda$; indeed, by
\eqref{eq:the_eigenvalue_map},
$\lambda=\gamma(z)=(q^z+q^{1-z})/(q+1)$, 
and, by \eqref{eq:Re(z)>=1/2},
$\gamma\colon J\to\mathC$ is bijective.
\end{proof}

\section{Continuity and surjectivity}\label {Sec:Surjectivity and hypercyclicity}

For each $\lambda>\rho$, denote by $\calH_\lambda$ the vector space of all 
$\lambda$-polyharmonic functions. We have seen in \eqref{eq:Poisson_representation_of_lambda-polyharmonic} that every  $f\in\calH_\lambda$ of order $m$ has a unique decomposition
$f(x) = \sum_{j=0}^{m-1} \int_{\partial T} \Poiss^{(j)}(x,\xi|\,\la)\,d\sigma_j(\xi)$. Let us introduce generalized Poisson transforms from distributions on $\partial T$ to functions on $T$ by setting, for $j\geqslant 0$,
\[
\Poisstr^{(j)} (\sigma)(x) =  \int_{\partial T} \Poiss^{(j)}(x,\xi|\,\la)\,d\sigma(\xi).
\]
In particular, $\Poisstr^{(0)}$ is the usual Poisson transformation at the eigenvalue $\lambda$.
\begin{corollary}\label{cor:Delta-gamma_I_surjective_on_span_of_polyharmonics} $\Delta -\gamma\;\mathI = P-\lambda\,\mathI$ is surjective on $\calH_{\lambda}$.
\end{corollary}
\begin{proof}
By \eqref{eq:action_of_P-lambda_on_lambda-polyharmonics},
 the generic element in $\calH_\lambda$ is of the type 
\begin{equation}\label{eq:the_shape_of_the_right_inverse}
\sum_{r=0}^{n} \Poisstr^{(r)} (\sigma_{r})= (P-\lambda\,\mathI) \, \sum_{r=0}^{n} (-1)^{r+1} \;\frac 1{r+1}\; \Poisstr^{(r+1)} (\sigma_r) .
\end{equation}

\end{proof}

Now let us define a norm on $\lambda$-polyharmonic functions as follows:
\begin{definition}\label{def:the_p-norm_of_lambda-polyharmonic_functions}
If $f$ is $\lambda$-polyharmonic, that is $f(x) = \sum_{j=0}^{m} \Poisstr^{(j)} (\sigma_j)(x)$,  we let
\begin{equation}\label {eq:the_norm_of_lambda-polyharmonic_functions}
\vertiii f= \sum_{j=0}^{m} {j!}\;\|\sigma_j\|\,.
\end{equation}
\end{definition}

Recall that $\lambda=\gamma(z)$. By \eqref {eq:inequality_between_norm_of_derivative_of_Poisson_kernel_and_Poisson_kernels_times_horsp.indices}, if $f\in\calH_\lambda$ has order $m$,
\begin{equation}\label{eq:pointwise_bound}
|f(x)|\leqslant \sum_{j=0}^{m} \max_{\xi\in\partial T} \bigl|\Poiss^{(j)} (x,\, \xi\, |\,\lambda)\bigr|\; \|\sigma_j\| \leqslant
 C_m(x,\lambda) \,  \sum_{j=0}^{m}  \|\sigma_j\|
\end{equation}
with $C_m(x,\lambda)= \mathcal{A}(\lambda)\,q^{|x|\Real z}\,\sqrt{m}\, |x|^m$. Note that, for every $\alpha\geqslant 0$, $\sqrt{m}\,\alpha^m/m!$ is a bounded sequence; therefore
$C_m(x,\lambda)\leqslant C(x,\lambda) \,m!$ for some constant $C(x,\lambda)$ (here, as already observed, $z$ is a function of $\lambda$).
Hence:

\begin{proposition}\label {prop:norm_convergence=>pointwise_convergence}
Every Cauchy sequence in the norm \eqref {eq:the_norm_of_lambda-polyharmonic_functions} has a pointwise limit.
\end{proposition}
\begin{proof}
Let $\{f_{n}\}$ be a Cauchy sequence in norm. As in \eqref{eq:Poisson_representation_of_lambda-polyharmonic}, let us write
$f_n(x) = \sum_{k=0}^{m_n} \int_{\partial T} \Poiss^{(k)}(x,\xi|\,\la)\,d\sigma^{(n)}_k(\xi)$. 
For simplicity, let us set $\sigma^{(n)}_k=0$ for $k>m_n$, whence $f_n(x) = \sum_{k=0}^{\infty} \int_{\partial T} \Poiss^{(k)}(x,\xi|\,\la)\,d\sigma^{(n)}_k(\xi)$.
Then
\begin{equation}\label{eq:norms_of_differences}
\vertiii{ f_n -f_j} = \sum_{k\geqslant 0} {k!}\;\|\sigma^{(n)}_k- \sigma^{(j)}_k\|. 
\end{equation}
Therefore, 
 for each $k$ we have $\|\sigma^{(n)}_k -  \sigma^{(j)}_k\| \leqslant \dfrac1{k!}\;\vertiii{ f_n-f_ j} $. Hence, for each $k$, the sequence $n\mapsto \sigma^{(n)}_k$ is a Cauchy sequence in the space of measures and so it converges to a limit measure $\sigma_k$. 
Note that $\{\vertiii{ f_n -f_j}\colon n,j\in\mathN\}$ is bounded because $\{f_n\}$ is a Cauchy sequence, hence, by \eqref{eq:norms_of_differences} and Fatou's Lemma, 
$ \sum_{k\geqslant 0} {k!}\;\|\sigma_k- \sigma^{(j)}_k\|$ is bounded.  Since
 $\sigma^{(j)}_k$ vanishes for $k>m_j$, it follows that $ \sum_{k\geqslant 0} {k!}\;\|\sigma_k\|$ is bounded.
 For every $x\in T$ and $n,k\in\mathN$, by \eqref{eq:pointwise_bound},
  $| f_n(x)-f_k(x)|\leqslant C_N(x,\lambda) \,  \sum_{j\geqslant 0}  \|\sigma^{(n)}_j-\sigma^{(k)}_j\|$. 
  Since $n\mapsto \sigma^{(n)}_k$ is a Cauchy sequence of measures on $\partial T$, by \eqref {eq:inequality_between_norm_of_derivative_of_Poisson_kernel_and_Poisson_kernels_times_horsp.indices}, $n \mapsto \Poiss^{(k)}(x,\xi |\,\la)\,d\sigma^{(n)}_k(\xi)$ is also a Cauchy sequence for every $x, \xi$ and $\lambda$,
    hence it converges  as $n\to\infty$. Then  $n\mapsto f_n(x)$ converges for every $x$ and $ \lambda$ to the limit  $f(x)=\sum_{k\geqslant 0} \int_{\partial T} \Poiss^{(k)}(x,\xi|\,\la)\,d\sigma_k(\xi)$. Indeed, this series converges for every $x$ and $\lambda$ because, by \eqref{eq:pointwise_bound} and the remarks following this inequality,  for every $N>0$  its $N$-tail satisfies
 \[
 |f(x)-f_n(x)| \leqslant   \sum_{k>N} C_k(x,\lambda) \, \|\sigma_k\| \leqslant  C(x,\lambda) \sum_{k>N} k! \, \|\sigma_k\|
 \]
 and we have seen that the series $ \sum k! \, \|\sigma_k\|$ converges. 
\end{proof}

\begin{definition}\label{def:the_right_space}
For $\lambda\notin \spec(P)$ we denote by $ Y_\lambda $ the closure of $\calH_\lambda$ in the norm of Definition \ref{def:the_p-norm_of_lambda-polyharmonic_functions}. By Proposition \ref{prop:norm_convergence=>pointwise_convergence}, $Y_\lambda$ is a space of functions on $T$.
\end{definition}

\begin{remark}\label {rem:surjectivity_of_P-lambda_on_a_dense_subspace}
The proof of Proposition \ref {prop:norm_convergence=>pointwise_convergence} shows that, 
if $f\in Y_\lambda$, then $f(x) = \sum_{j=0}^{\infty} \Poisstr^{(j)} (\sigma_j)(x)$
with  $\sum_{j=0}^{\infty} j!\,\|\sigma_j\| <\infty$, hence 
$\|\sigma_j\|=o(1/j!)$, and in particular the sequence $\{\|\sigma_j\|\}$ is bounded.
Terefore $ Y_\lambda $ is spanned by $\lambda$-polyharmonic functions
 whose boundary measures $\sigma_j$ in the Poisson representation \eqref{eq:Poisson_representation_of_lambda-polyharmonic} have norms  bounded by 1.

\end{remark}
\begin{proposition}\label{prop:continuity}
For every $\lambda$, $P-\lambda\mathI$ is continuous in the norm of $Y_\lambda$, and $\| P-\lambda\,\mathI\|\leqslant 1$.
\end{proposition}
\begin{proof}
Since $\calH_\lambda$ is dense in $Y_\lambda$ it is enough to consider $f\in\calH_\lambda$, that is,  $f(x) = \sum_{j=0}^{m} \Poisstr^{(j)} (\sigma_j)$. Then by \eqref {eq:action_of_P-lambda_on_lambda-polyharmonics}
\[
(P-\lambda\,\mathI) f = \sum_{j=1}^{m}  (-1)^j j \,\Poisstr^{(j-1)} (\sigma_j) = \sum_{j=0}^{m-1}  (-1)^{j+1} (j+1) \,\Poisstr^{(j)} (\sigma_{j+1}),
\]
 hence
\[
\vertiii{(P-\lambda\,\mathI)f } = \sum_{j=0}^{m-1} (j+1)\;{j!}\;\|\sigma_{j+1}\| =
 \sum_{j=1}^{m} \; j!\,\|\sigma_{j}\| \leqslant \vertiii f .
\] 
\end{proof}

It would be nice to find a Banach space where $P-\lambda\,\mathI$ is an open operator, more specifically such that $P-\lambda\,\mathI$ is bounded below, i.e., $\| (P-\lambda\,\mathI) f \| \geqslant \| f \|$. This, however, is not possible, because $ (P-\lambda\,\mathI) $ has a non-trivial kernel; indeed, $x\mapsto \Poiss(x,\xi|\,\la)$ is a $\lambda$-eigenfunction of $P$ that belongs to $Y_\lambda$. Therefore we shall look at the quotient $Y_\lambda/\ker(P-\lambda\,\mathI)$ and show that here  $P-\lambda\,\mathI$ is isometric. More precisely, for each $h$ we shall consider a function $f$ such that $(P-\lambda\,\mathI) f =h$ and $\|f\| = \|h\|$, defined as follows.

By Corollary \ref{cor:Delta-gamma_I_surjective_on_span_of_polyharmonics}, $h(x)=\sum_{j=0}^{n}  \Poisstr^{(j)} (\sigma^{(h)}_j) $ for every $\lambda$-polyharmonic function $h$, and the set of functions $f$ such that 
$(P-\lambda\,\mathI)f = h$ form the hyperplane
\[
\Bigl\{ f = \sum_{j=0}^{n} (-1)^{j+1} \;\frac1{j+1} \;\Poisstr^{(j+1)}(\sigma^{(h)}_j) \mod \ker (P-\lambda\,\mathI)\Bigr\},
\]
i.e., they are unique modulo the kernel of $P-\lambda\,\mathI$; we shall choose $f$ such that this additional term is null. That is, for each $\lambda$-polyharmonic $h$ we choose $f_h= \sum_{j=0}^{n}  (-1)^{j+1} \;\frac1{j+1} \;\Poisstr^{(j+1)}\,(\sigma^{(h)}_j)$.
Then $(P-\lambda\,\mathI)f_h=h$, and by \eqref{eq:the_norm_of_lambda-polyharmonic_functions}
\begin{equation}\label {eq:P-lambda_is_open_on_a_suitable_subspace}
\vertiii{ f _h} = \sum_{j=0}^{n} \;j!\;\|\sigma^{(h)}_j\| = \vertiii h. 
\end{equation}

\begin{lemma}\label{lemma:surjectivity}
$P-\lambda\,\mathI$ is surjective on $Y_\lambda $ for $\lambda\notin\spec(P)$.
\end{lemma}
\begin{proof}
We have shown in Remark \ref{rem:surjectivity_of_P-lambda_on_a_dense_subspace} that $P-\lambda\,\mathI$ is surjective on $\lambda$-polyharmonic functions. Now take any element $h\in Y_\lambda $. Then $h$ is  the limit of a  sequence of $\lambda$-polyharmonic functions $h_n$ in the norm of $Y_\lambda $. By \eqref{eq:P-lambda_is_open_on_a_suitable_subspace}, $\{f_{h_n}\}$ is a Cauchy sequence in this norm, hence it converges in norm to a function $f\in Y_\lambda $. Now $h_n = (P-\lambda\,\mathI) f_{h_n}$, and $h_n \to h$ and $f_{h_n} \to f$ in norm. 
By Proposition \ref{prop:norm_convergence=>pointwise_convergence}, $f_{h_n}(x) \to f(x)$ for every $x\in T$.
Since
$P-\lambda\,\mathI$ is an operator of range 1 (that is, $(P-\lambda\,\mathI) f (x)$ depends only on the values of $f$ at the vertex $x$ and its neighbors), the fact that $f_{h_n}\to f$ pointwise implies that $h_n(x)=(P-\lambda\,\mathI) f_{h_n} (x) \to (P-\lambda\,\mathI) f (x)$ for every $x$, so $(P-\lambda\,\mathI) f (x)$ is the pointwise limit of the sequence $\{h_n\}$. Since $h_n$ converges to $h$ in norm, this pointwise limit is $h$, again by Proposition \ref{prop:norm_convergence=>pointwise_convergence}. 
\end{proof}

\section{Hypercyclicity}

\begin{definition}[Hypercyclic operators] \label{def:hypercyclic}
A vector $f$ is hypercyclic with respect to an operator $\T$ on some Banach space $Y$ if the set $\{\T^nf\}$ is dense.   $\T$ is \emph{hypercyclic} if it has a hypercyclic vector.  
\end{definition}

\begin{proposition}
For $\lambda\notin\spec(P)$. the operator $P-\lambda\,\mathI$ is  not hypercyclic on $Y_\lambda$.
\end{proposition}
\begin{proof}
For each $f\in Y_\lambda$ choose a sequence $f_j\in\calH_\lambda$ such that $f_j$ converges to $f$ in norm. Let $m_j$ be the order of the $\lambda$-polyharmonic function $f_j$. For an arbitrary $\epsilon>0$ choose $j$ such that $\vertiii{f-f_j}<\epsilon$. Then, by Proposition \ref{prop:continuity}, for every $n\in\mathN$, $\vertiii{(P-\lambda\,\mathI)^n (f-f_j)} <  \epsilon$, hence $\vertiii{(P-\lambda\,\mathI)^n f} <
\vertiii{(P-\lambda\,\mathI)^n f_j} + \epsilon$. Since $(P-\lambda\,\mathI)^n f_j = 0$ for $n\geqslant m_j$, it follows that $(P-\lambda\,\mathI)^n f$ tends to zero in $Y_\lambda$, and so $P-\lambda\,\mathI$ is not hypercyclic.
\end{proof}

However:

\begin{theorem}\label{theo:hypercyclicity}
Let $\lambda\notin\spec(P)$.  Then $\T=e^{t(P-\lambda \mathI)} = e^{t(\Delta-\gamma \, \mathI)}$ is hypercyclic on $Y_\lambda $ for every $t\neq 0$.
\end{theorem}
\begin{proof} 
The $\lambda$-polyharmonic functions of $P$ of order $n$ form the kernel of $(P-\lambda\mathI)^n$: therefore $\calH_\lambda=\bigcup_{n\geqslant 0} \ker \left(P-\lambda\mathI\right)^n$, and this space is dense in $Y_\lambda$. 
The operator $P-\lambda\mathI$ is surjective onto  $Y_\lambda $ by Lemma \ref{lemma:surjectivity}, and is bounded by Remark \ref{rem:surjectivity_of_P-lambda_on_a_dense_subspace}\,$(ii)$.
Then it follows from \cite{HS}*{Theorem 1} that $\psi(\T)$ is hypercyclic if $\psi$ is a non-constant function analytic in a neighborhood of $\s$ with $|\psi(0)|=1$ and $0\notin \psi(\s)$.
Hence the operator $\T=e^{t(\Delta-\gamma \, \mathI)}$ is hypercyclic for $t\neq 0$, since $0$ is not in the image of the exponential map $z\mapsto e^{tz}$.
\end{proof}

\end{document}